\documentclass[12pt,a4paper]{amsart}
\usepackage{amssymb,amsmath,amsfonts}

\headsep=1.2500cm \numberwithin{equation}{section}
\newcommand{\set}[1]{\left\{#1\right\}}

\newtheorem{Theorem}{Theorem}[section]
\newtheorem{Proposition}[Theorem]{Proposition}
\newtheorem{cor}[Theorem]{Corollary}
\newtheorem{lemma}[Theorem]{Lemma}
\theoremstyle{remark}

\newtheorem{Remark}[Theorem]{Remark}

\begin{document}

\title{Approximate biflatness and Johnson pseudo-contractibility of some Banach algebras}

  \author[A. Sahami]{A. Sahami}
  \email{a.sahami@ilam.ac.ir}
  
  \address{Department of Mathematics,
  	Faculty of Basic Sciences, Ilam University P.O. Box 69315-516 Ilam,
  	Iran.}

\author[M. R.  Omidi]{M. R.  Omidi}
\email{m.omidi@kut.ac.ir}

\address{ Department of Basic Sciences, Kermanshah University of Technology, Kermanshah, Iran}

\author[E.  Ghaderi]{E. Ghaderi}
\email{eg.ghaderi@uok.ac.ir}

\address{ Department of Mathematics,  University of Kurdistan,  Pasdaran boulevard, Sanandaj 66177--15175,   P. O. Box 416, Iran.}

\author[H. Zangeneh]{H. Zangeneh}
\email{hamzeh.za@gmail.com}
\address{Department of Mathematics,
	Faculty of Basic Sciences, Ilam University P.O. Box 69315-516 Ilam,
	Iran.}
\keywords{Approximate biflatness, Johnson pseudo-contractibility, Lipschitz algebra, triangular Banach algebra.}

\subjclass[2010]{ Primary 46M10, 46H20, Secondary 46H05.}

\maketitle
%-------------------------------------------------------------

%%%%%%%%%%%%%%%%%%%%%%%%%%%%%%%%%%%%%%%%%%%%%%%%%%%%%%%%%%%%%%%%%%%%%%%%%
\begin{abstract}
	In this paper, we study  the structure of  Lipschitz algebras under the notions of approximate biflatness and Johnson pseudo-contractibility.
We show that for  a compact  metric space $X,$  the Lipschitz algebras  $Lip_{\alpha}(X)$ and $\ell ip_{\alpha}(X)$  are approximately biflat if and only if $X$ is finite, provided that $0<\alpha<1$. We give an enough and sufficient condition that a vector-valued Lipschitz algebras is Johnson pseudo-contractible for each $\alpha>0.$
We also show that some triangular Banach algebras are not approximately biflat.
\end{abstract}
\section{Introduction and preliminaries}
A Banach algebra $A$ is called amenable if there exists a bounded net $(m_{\alpha})$ in $A\otimes_{p}A$ such that $a\cdot m_{\alpha}-m_{\alpha}\cdot a \rightarrow 0$ and $\pi_{A}(m_{\alpha})a\rightarrow a $ for every $a\in A,$ where $\pi_{A}:A\otimes_{p}A\rightarrow A$ is  the product morphism given by  $\pi_{A}(a\otimes b)=ab.$ Johnson showed that for a locally compact group $G$, $L^{1}(G)$ is amenable if and only if $G$ is amenable. For more information about the history of amenability, the reader refers to \cite{run}.

An important notion of homological theory related to amenability is biflatness. In fact a Banach algebra $A$ is called biflat, if there exists a bounded $A$-bimodule morphism $\rho:(A\otimes_{p}A)^*\rightarrow A^*$ such that $\rho\circ\pi^*_{A}=id_{A^*}$.  It is well-known that a Banach algebra $A$ is amenable if and only if $A$ is biflat and $A$ has a bounded approximate identity.

Motivated by these considerations,    Samei {\it et al.} introduced in \cite{sam}  the approximate version of biflatness. Indeed a Banach algebra $A$ is approximately biflat if there exists a net of $A$-bimodule morphism $(\rho_{\alpha})$ from $(A\otimes_{p}A)^*$ into $A^*$ such that $\rho_{\alpha}\circ\pi^*_{A}\xrightarrow {W^*OT}id_{A^*}$, where $W^*OT$ stands for the  weak star operator topology. They studied approximate biflatness of the Segal algebras and the Fourier algebras. The Lipschitz algebras are concrete Banach algebras, see \cite{Lipschitz}. These algebras are rely upon the metric spaces.  In this paper, we characterize  approximate biflatness of  Lipschitz algebras and we show that for a compact  metric space $X,$  the Lipschitz algebras  $Lip_{\alpha}(X)$ and $\ell ip_{\alpha}(X)$  are approximately biflat if and only if $X$ is finite, provided that $0<\alpha<1$. We also study the Johnson pseudo-contractibility of vector-valued Lipschitz algebras and we investigate  the approximate biflatness of some triangular Banach algebras.

We present some standard notations and definitions that we shall need
in this paper. Let $A$ be a Banach algebra. Throughout this work,
the character space of $A$ is denoted by $\Delta(A)$, that is, the set all
non-zero multiplicative linear functionals on $A$.
 For each
$\phi\in\Delta(A)$ there exists a unique extension $\tilde{\phi}$ to
$A^{**}$ which is defined by $\tilde{\phi}(F)=F(\phi)$. It is easy to
see that $\tilde{\phi}\in\Delta(A^{**})$.
 The projective
tensor product
$A\otimes_{p}A$ is a Banach $A$-bimodule via  the following actions
$$a\cdot(b\otimes c)=ab\otimes c,~~~(b\otimes c)\cdot a=b\otimes
ca\hspace{.5cm}(a, b, c\in A).$$ 
Let $X$ and $Y$ be Banach $A$-bimodules. The linear map $T:X\rightarrow Y$ is called $A$-bimodule morphism, if 
$$T(a\cdot x)=a\cdot T(x),\quad T(x\cdot a)=T(x)\cdot a,\qquad (a\in A,x\in X).$$
%------------------------------------------------------------------------------------------------------------------------------------------
%%%%%%%%%%%%%%%%%%%%%%%%%%%%%%%%%%%%%%%%%%%%%%%%%%%%%%%%%%%%%%%%%%%%%%%%%%%%%%%%%%%%%%%%%%%%%%%%%%%%%%%%%%%%%%%%%%%%%%%%%%%%%%%%%%%%%%%%%%%
%------------------------------------------------------------------------------------------------------------------------------------------
\section{Johnson pseudo-contractibility and approximate biflatness}
We recall that the Banach algebra $A$ is Johnson pseudo-contractible, 
if  there  exists a not necessarily bounded net $(m_{\alpha})$ in $(A\otimes_{p}A)^{**}$ such that $a\cdot m_{\alpha}=m_{\alpha}\cdot a $ and $\pi^{**}_{A}(m_{\alpha})a\rightarrow a $ for each $a\in A,$ see \cite{sah new amen} and \cite{sah new amen1}.
\begin{Theorem}
Let $A$ be a Johnson pseudo-contractible Banach algebra. Then $A$ is approximately biflat.
\end{Theorem}
\begin{proof}
Suppose that $A$ is  a Johnson pseudo-contractible Banach algebra. Then there  exists a net $(m_{\alpha})$ in $(A\otimes_{p}A)^{**}$ such that $a\cdot m_{\alpha}=m_{\alpha}\cdot a $ and $\pi^{**}_{A}(m_{\alpha})a\rightarrow a $ for each $a\in A.$
Define $\theta_{\alpha}(a)=a\cdot m_{\alpha}$. Clearly  $(\theta_{\alpha})_{\alpha }$ is a net of $A$-bimodule morphisms from $A$ into $(A\otimes_{p}A)^{**}$ such that $\pi_{A}^{**}\circ\theta_{\alpha}(a)\rightarrow a,$ for each $a\in A.$ Put $\rho_{\alpha}=\theta^{*}_{\alpha}|_{(A\otimes_{p}A)^*}:(A\otimes_{p}A)^*\rightarrow A^*$. It is easy to see that $(\rho_{\alpha})_{\alpha}$ is a net of $A$-bimodule morphisms. We claim that $$\rho_{\alpha}\circ \pi^{*}_{A}\xrightarrow{W^*OT} id_{A^*}.$$ To see this, let $a\in A$ and $f\in A^*. $
\begin{equation*}
\begin{split}
<\rho_{\alpha}\circ \pi^{*}_{A}(f),a>-<a,f>&=<\theta^{*}_{\alpha}|_{(A\otimes_{p}A)^*}\circ \pi^*_{A}(f),a>-<a,f>\\
&=<\theta^{***}_{\alpha}\circ \pi^*_{A}(f),a>-<a,f>\\
&=< \pi^*_{A}(f),\theta^{**}_{\alpha}(a)>-<a,f>\\
&=< \pi^*_{A}(f),\theta_{\alpha}(a)>-<a,f>\\
&=< \theta_{\alpha}(a),\pi^*_{A}(f)>-<a,f>\\
&=< \pi^{**}_{A}\circ \theta_{\alpha}(a),f>-<a,f>\rightarrow 0.
\end{split}
\end{equation*}
It follows that $A$ is approximately biflat.
\end{proof}
\begin{Remark}
The converse of above theorem is not always true. To see this, suppose that    $S$ is  the  left zero semigroup with $|S|\geq
2$, that is, a semigroup with product $st=s$ for all $s,t\in S.$
Then the related  semigroup algebra $\ell^{1}(S)$ has the following product
$$fg=\phi_{S}(f)g,\quad f,g\in \ell^{1}(S),$$
where $\phi_{S}$ is denoted for the augmentation character on $\ell^{1}(S)$.
Define $\rho:\ell^{1}(S)\rightarrow (\ell^{1}(S)\otimes_{p}\ell^{1}(S))^{**}$ by $\rho(f)=f_{0}\otimes f$. Clearly $\rho$ is a bounded $\ell^{1}(S)$-bimodule morphism
which $\pi^{**}_{\ell^{1}(S)}\circ \rho(f)=f$, for each $f\in \ell^{1}(S).$ Applying \cite[Lemma 4.3.22]{run}, $\ell^{1}(S)$ becomes biflat. So $\ell^{1}(S)$ is approximately biflat. We claim that $\ell^{1}(S)$ is not Johnson pseudo-contractible. We assume conversely that $\ell^{1}(S)$ is 
  Johnson pseudo-contractible.
It is easy to see that $\ell^{1}(S)$ has an approximate identity, say $(e_{\alpha})$. Consider 
$$\phi_{S}(e_{\alpha})\rightarrow 1, \quad
e_{\alpha}f-fe_{\alpha}=
\phi_{S}(e_{\alpha})f-\phi_{S}(f)e_{\alpha}\rightarrow 0\quad (f\in
\ell^{1}(S)).$$
It follows that $f-\phi_{S}(f)e_{\alpha}\rightarrow
0$ for each $f\in \ell^{1}(S)$. Since  there exist at least  two different elements
$s_{1}$ and $s_{2}$ in $S$, replace two distinct elements  $\delta_{s_{1}}$ and  $\delta_{s_{2}}$ of $\ell^{1}(S)$
with $f$ in
$f-\phi_{S}(f)e_{\alpha}\rightarrow 0$. It follows that
$\delta_{s_{1}}=\delta_{s_{2}}$, so $s_{1}=s_{2}$ which is
a contradiction.
\end{Remark}
It is still open, whether the approximately biflatness of $A$ implies the Johnson pseudo-contractibility of $A$.
\begin{lemma}\label{net for app bi}
	Let $A$ be an approximately biflat Banach algebra with a central approximate identity. Then there is a net $(m_{\gamma})$ in $(A\otimes_{p}A)^{**}$ such that $$a\cdot m_{\gamma}=m_{\gamma}\cdot a,\quad \pi_{A}^{**}(m_{\gamma})a\xrightarrow{w^*} a,\qquad (a\in A).$$
\end{lemma}
\begin{proof}
Suppose that $A$ is an approximately biflat Banach algebra with a central approximate identity, say $(e_{\beta})_{\beta\in J}$.
Then there exists a net of $A$-bimodule morphism $(\rho_{\alpha})_{\alpha\in I}$ from $(A\otimes_{p}A)^*$
into $A^*$ such that $\rho_{\alpha}\circ \pi^*_{A}\xrightarrow{W^*OT}id_{A^*}.$ Set $m^{\beta}_{\alpha}=\rho_{\alpha}^{*}(e_{\beta})$. Since $(\rho_{\alpha}^{*})$ is a net of $A$-bimodule morphism, we have 
\begin{equation*}
\begin{split}
a\cdot m^{\beta}_{\alpha}=a\cdot \rho_{\alpha}^{*}(e_{\beta})=\rho_{\alpha}^{*}(ae_{\beta})
=\rho_{\alpha}^{*}(e_{\beta}a)
=\rho_{\alpha}^{*}(e_{\beta})\cdot a
=m^{\beta}_{\alpha}\cdot a,
\end{split}
\end{equation*}
for each $\alpha\in I$, $\beta \in J$ and $a\in A.$
Also for each $a\in A$ and $\phi\in A^*$, we have 
\begin{equation}\label{eq}
\begin{split}
\lim_{\beta}\lim_{\alpha}<\phi, \pi_{A}^{**}(m^{\beta}_{\alpha})\cdot a>&=\lim_{\beta}\lim_{\alpha}<\phi\cdot a, \pi_{A}^{**}(m^{\beta}_{\alpha})>\\
&=\lim_{\beta}\lim_{\alpha}<\phi\cdot a, \pi_{A}^{**}(\rho_{\alpha}^{*}(e_{\beta}))>\\
&=\lim_{\beta}\lim_{\alpha}<\rho_{\alpha}\circ \pi^*_{A}(\phi\cdot a), e_{\beta})>\\
&=\lim_{\beta} <\phi\cdot a, e_{\beta})>\\
&=\lim_{\beta} <\phi,a e_{\beta})>=<a,\phi>.
\end{split}
\end{equation}
Set $E=J\times I^{J}$, where $I^{J}$ is the set of
all functions from $J$ into $I$. Consider the product ordering on $E$ as follow  $$(\beta,\alpha)\leq_{E}
(\beta^{'},\alpha^{'})\Leftrightarrow \beta\leq_{J} \beta^{'},
\alpha\leq_{I^{J}}\alpha^{'}\qquad (\beta,\beta^{'}\in J,\quad
\alpha,\alpha^{'}\in I^{J}),$$ here  $
\alpha \leq_{I^{J}}\alpha^{'}$ means that $\alpha(d)\leq_{I}
\alpha^{'}(d)$ for each $d\in J$. Suppose that
$\gamma=(\beta,\alpha_{\beta})\in E$ and $m_{\gamma}=\rho^*_{\alpha_{\beta}}(e_{\beta})\in (A\otimes_{p}A)^{**}$.
Now using
iterated limit theorem \cite[page 69]{kel} and the equation (\ref{eq}), we have 
$$a\cdot m_{\gamma}=m_{\gamma}\cdot a,\quad \pi_{A}^{**}(m_{\gamma})a\xrightarrow{w^*} a,\qquad (a\in A).$$
\end{proof}

Let $A$ be a Banach algebra and $\phi\in\Delta(A)$.
An element $m\in A^{**}$
that satisfies $am=\phi(a)m$ and $\tilde{\phi}(m)=1,$ is called $\phi$-mean.  Suppose that
$m\in A^{**}$ is a  $\phi$-mean for $A$.  Since
$||\phi||=1,$ we have $||m||\geq 1.$ So for $C\geq 1$,
$A$ is called
$C$-$\phi$-amenable if  $A$ has a $\phi$-mean $m$ which $||m||\leq
C.$ Also $A$ is called
$C$-character amenable if  $A$ has a bounded right approximate identity and has  $\phi$-mean $m$ which $||m||\leq
C$, for eavery $\phi\in\Delta(A),$  see \cite{kan} and \cite{Hu}.
\begin{Proposition}\label{c character amen}
	Let $A$ be a Johnson pseudo-contractible Banach algebra and $\Delta(A)\neq \emptyset$. If $A$ has a right identity, then $A$ is $C$-character amenable.
\end{Proposition}
\begin{proof}
Similar to the proof of \cite[Lemma 3.5]{sah new amen1}.	
\end{proof}
\begin{lemma}\label{app left}
Let $A$ be an  approximately biflat Banach algebra with an
identity and $\Delta(A)\neq \emptyset$. Then $A$ is
$\phi$-amenable for every $\phi\in\Delta(A)$.
\end{lemma}
\begin{proof}
Suppose that $A$ is an  approximately biflat with an identity $e$. Then by Lemma \ref{net for app bi}, there exists  a net
$(m_{\alpha})$ in $(A\otimes_{p}A)^{**}$ such that $a\cdot
m_{\alpha}=m_{\alpha}\cdot a$ and
$\pi^{**}_{A}(m_{\alpha})a\xrightarrow{w^*} a$ for every $a\in A.$
So for
every $\epsilon>0$ there exists $\alpha^{\phi}_{\epsilon}$ such that
$$|\tilde{\phi}\circ \pi^{**}_{A}(m_{\alpha^{\phi}_{\epsilon}})-1|=|\tilde{\phi}\circ \pi^{**}_{A}(m_{\alpha^{\phi}_{\epsilon}})-\tilde{\phi}(e)|=|\pi^{**}_{A}(m_{\alpha^{\phi}_{\epsilon}})e(\phi)-e(\phi)|<\epsilon$$
 and $a\cdot
m_{\alpha^{\phi}_{\epsilon}}=m_{\alpha^{\phi}_{\epsilon}}\cdot a$.  Let $T:A\otimes_{p}A\rightarrow A$ be a map
defined by $T(a\otimes b)=\phi(b)a$ for every $a,b\in A.$ Since $\tilde{\phi}\circ T^{**}=\tilde{\phi}\circ\pi^{**}_{A}$, it follows that 
\begin{equation}\label{eq2}
|\tilde{\phi}\circ
T^{**}(m_{\alpha^{\phi}_{\epsilon}})-1|=|\tilde{\phi}(\pi_{A}^{**}(m_{\alpha^{\phi}_{\epsilon}}))-1|<\epsilon.
\end{equation} 
As we know that  $T^{**}$ is a $w^{*}$-continuous map, thus
$$T^{**}(a\cdot F)=a\cdot T^{**}(F),\quad \phi(a)T^{**}( F)=T^{**}(F\cdot a),\qquad (a\in A, F\in (A\otimes_{p}A)^{**}).$$
Then
$$a\cdot T^{**}(m_{\alpha^{\phi}_{\epsilon}})=T^{**}(a\cdot m_{\alpha^{\phi}_{\epsilon}})=T^{**}( m_{\alpha^{\phi}_{\epsilon}}\cdot a)=\phi(a)T^{**}( m_{\alpha^{\phi}_{\epsilon}})$$
for every $a\in A.$ 
Replacing  $T^{**}(m_{\alpha^{\phi}_{\epsilon}})$ by
$\frac{T^{**}(m_{\alpha^{\phi}_{\epsilon}})}{\tilde{\phi}\circ T^{**}(m_{\alpha^{\phi}_{\epsilon}})}$, we
may suppose  that
$$aT^{**}(m_{\alpha^{\phi}_{\epsilon}})=\phi(a)T^{**}(m_{\alpha^{\phi}_{\epsilon}}),\quad \tilde{\phi}\circ T^{**}(m_{\alpha^{\phi}_{\epsilon}})=1,$$
for every $a\in A.$ It shows that $A$ is left $\phi$-amenable.
\end{proof}
%------------------------------------------------------------------------------------------------------------------------------------------
%%%%%%%%%%%%%%%%%%%%%%%%%%%%%%%%%%%%%%%%%%%%%%%%%%%%%%%%%%%%%%%%%%%%%%%%%%%%%%%%%%%%%%%%%%%%%%%%%%%%%%%%%%%%%%%%%%%%%%%%%%%%%%%%%%%%%%%%%%%
%------------------------------------------------------------------------------------------------------------------------------------------
\section{Applications to Lipschitz algebras}

Let $X$ be a  metric space and $\alpha>0$. Also let $(E,||\cdot||)$ be a Banach space.  Set
$$Lip_{\alpha}(X,E)=\{f:X\rightarrow
E:\text{$f$ is bounded and }p_{\alpha,E}(f)<\infty\},$$ where
$$p_{\alpha,E}(f)=\sup\{\frac{||f(x)-f(y)||}{d(x,y)^{\alpha}}:x,y\in
X,x\neq y\}$$ and also $$\ell ip_{\alpha}(X,E)=\{f\in Lip_{\alpha}(X,E)
:\frac{||f(x)-f(y)||}{d(x,y)^{\alpha}}\rightarrow 0\quad
\text{as}\quad d(x,y)\rightarrow 0\}.$$ Define
$$||f||_{\alpha,E}=||f||_{\infty,E}+p_{\alpha,E}(f),$$ where $||f||_{\infty,E}=\sup_{x\in X}||f(x)||$. For each Banach algebra $E$,
with the pointwise multiplication and norm
$||\cdot||_{\alpha,E}$, $Lip_{\alpha}(X,E)$ and $\ell
ip_{\alpha}(X,E)$ become Banach algebras.
Also we denote $Lip_{\alpha}(X)$ for $Lip_{\alpha}(X,\mathbb{C})$ and $\ell
ip_{\alpha}(X)$ for $\ell
ip_{\alpha}(X,\mathbb{C})$, respectively.

If $X$ is a compact metric space, it is well-known that each
non-zero multiplicative  linear functional on $Lip_{\alpha}(X)$ (also on
$\ell ip_{\alpha}(X)$  ) has a form $\phi_{x}$ for some $x\in X$, where
$\phi_{x}(f)=f(x)$ for every $x\in X$. For further information about the
Lipschitz algebras see \cite{bade}, \cite{Lipschitz}, \cite{gord1} and \cite{dashti}. A metric space $(X,d)$ is called uniformly discrete, if there exists $\epsilon>0$ such that $d(x,y)>\epsilon$ for every $x,y\in X$ with $x\neq y.$
\begin{Theorem}
	Let $(X,d)$ be a metric space and $\alpha>0$ and let $E$ be a Banach algebra with a right identity with $\Delta(E)\neq \emptyset$. Suppose that $\ell ip_{\alpha}(X,E)$ or $Lip_{\alpha}(X,E)$ separates the elements of $X$ and $E.$ If $\ell ip_{\alpha}(X,E)$ or $Lip_{\alpha}(X,E)$ is Johnson pseudo-contractible, then $X$ is uniformly discrete and $E$ is Johnson pseudo-contractible.
\end{Theorem}
\begin{proof}
	Let $A$ be $\ell ip_{\alpha}(X,E)$ or $Lip_{\alpha}(X,E)$. Since $E$ has a right identity, $A$ has a right identity. Using Johnson pseudo-contractibility of $A$ and Proposition \ref{c character amen}, we have $A$ is $C$-character amenable.
	By \cite[Lemma 3.1]{Biya} $X$ is uniformly discrete. Let $x_{0}\in X$. Define $\phi_{x_{0}}:A\rightarrow E$ by $\phi_{x_{0}}(f)=f(x_{0})$. Clearly $\phi_{x_{0}}$ is a  homomorphism and onto bounded linear map. Since $A$ is Johnson pseudo-contractible by \cite[Proposition 2.9]{sah new amen1}, $E$
is Johnson pseudo-contractible.	
\end{proof}
\begin{Proposition}
	Let $(X,d)$ be a metric space and $\alpha>0$ and let $E$ be a Banach algebra with an identity which $\Delta(E)\neq \emptyset$. Suppose that $A= \ell ip_{\alpha}(X,E)$ or $A=Lip_{\alpha}(X,E)$ separates the elements of $X$ and $E.$ 
	Then the following statements are equivalent:
	\begin{enumerate}
		\item [(i)] $A$ is Johnson pseudo-contractible.
		\item [(ii)] $X$ is uniformly discrete.
		\item [(iii)] $A$ is amenable.
	\end{enumerate}
\end{Proposition}
\begin{proof}
(i)$\Leftrightarrow$(ii)	Since $E$ is unital,  by \cite[Theorem 1.1]{mehdi} $E$ is amenable. Clearly $A$ is a unital Banach algebra. Also by \cite[Theorem 1.1]{mehdi}  , Johnson pseudo-contractibility of $A$ implies that $A$ is amenable. Applying \cite[Theorem 3.4]{Biya} finishes the proof.  

(ii)$\Leftrightarrow$(iii) It is clear by \cite[Theorem 3.4]{Biya}.
\end{proof}
\begin{Theorem}
Let $X$ be a compact metric space and let $A$ be $Lip_{\alpha}(X)$ or
$\ell ip_{\alpha}(X)$  with $0<\alpha<1$. Then the following statements are equivalent
\begin{enumerate}
\item [(i)] $A$ is approximately biflat.
\item [(ii)] $X$ is finite.
\item [(iii)] $A$ is amenable.
\end{enumerate}
\end{Theorem}
\begin{proof}
(i)$\Rightarrow$(ii) Let  $A$  be an approximately biflat Banach algebra. Since $A$ has an identity, by  Lemma \ref{app left},  $A$ is $\phi$-amenable for every $\phi\in\Delta(A)$.  On the other hand the existence of an identity for $A$ follows that $A$ is character amenable. Suppose, towards a contradiction, that $X$ is infinite and $x_{0}\in X$ is not isolated point of $X$. Since by \cite[ Theorem 4.4.30(iv)]{dales}, $\ker \phi_{x_{0}} $ does not have a bounded approximate identity, \cite[Lemma 3.3]{Hu} implies that $A$ is not character amenable, which is impossible. It implies that $X$ is discrete, so $X$ is finite.

(ii)$\Rightarrow$(iii) See \cite[Theorem 3]{gord}.

(iii)$\Rightarrow$(i)  Suppose that $A$ is amenable. Then there  exists an element $M\in (A\otimes_{p}A)^{**}$ such that $a\cdot M=M\cdot a$ such that $\pi^{**}_{A}(M)a=a,$ for each $a\in A$. Define $\rho:A\rightarrow (A\otimes_{p}A)^{**}$ by $\rho(a)=a\cdot M$. It is easy to see that $\rho$ is a bounded $A$-bimodule morphism and $A$ is biflat, see \cite[Lemma 4.3.22]{run}. It follows that $A$ is approximately biflat.
\end{proof}
\section{Applications to Triangular Banach algebras}
Let $A$ be a Banach algebra and $\phi\in\Delta(A)$. Suppose that $X$
is a Banach left  $A$-module. A non-zero linear functional $\eta\in
X^{*}$ is called left $\phi$-character if $\eta(a\cdot
x)=\phi(a)\eta(x)$ and it is called  right $\phi$-character  if $\eta(x\cdot
a)=\phi(a)\eta(x)$. A left and a right $\phi$-character is called
$\phi$-character. Note that if $A$ is a Banach algebra and
$\phi\in\Delta(A)$, then $\phi\otimes\phi$  on $A\otimes_{p}A$ (defined by $\phi\otimes\phi(a\otimes b)=\phi(a)\phi(b)$) and $\tilde{\phi}$ on
$A^{**}$ are $\phi$-characters.

Let $A$ and $B$ be  Banach algebras and let $X$ be a Banach
$(A,B)$-module. That is, $X$ is a Banach left
$A$-module and a Banach right $B$-module that satisfy $(a\cdot x)\cdot b=a\cdot(x\cdot b)$ and $||a\cdot x\cdot b||\leq ||a||||x||||b||$ for every $a\in A$, $b\in B$ and $x\in X$.
Consider 

$$T=Tri(A,B,X)=\set{\left(\begin{array}{cc} a&x\\
	0&b\\
	\end{array}\right):a\in A,x\in X,b\in B},$$
with the usual matrix operations and  $$||\left(\begin{array}{cc} a&x\\
0&b\\
\end{array}\right)||=||a||+||x||+||b||\quad(a\in A,x\in X, b\in B)$$  $T$ becomes a Banach algebra which is called triangular Banach algebra.
Let $\phi\in\Delta(B)$. We define a character  $\psi_{\phi}\in\Delta(T)$  via  $\psi_{\phi}\left(\begin{array}{cc} a&x\\
0&b\\
\end{array}
\right)=\phi(b)$ for every $a\in A$, $b\in B$ and $x\in X$.
\begin{Theorem}\label{main}
	Let $T=Tri(A,B,X)$ be a triangular Banach algebra such that $A$ and $B$ have a central approximat identity ($\Delta(B)\neq \emptyset$).
 If one of the
	followings  hold
	\begin{enumerate}
		\item [(i)] B is not left $\phi$-amenable,
		\item [(ii)] $X$ has a right $\phi$-character,
	\end{enumerate}
	then $T$ is not approximately biflat.
\end{Theorem}
\begin{proof}
	Suppose, in contradiction, that
	$T$ is approximately biflat. Since $T$ has a central approximate identity, by similar  argument as in  the	
	 Lemma \ref{app left}, 
	 $T$ is left $\psi_{\phi}$-amenable. Clearly $I=\left(\begin{array}{cc} 0&X\\
	0&B\\
	\end{array}
	\right)$ is a closed ideal of $T$ and $\psi_{\phi}|_{I}\neq 0$,
	then by \cite[Lemma 3.1]{kan}
	 $I$ is left
	$\psi_{\phi}$-amenable. Thus by \cite[Theorem 1.4]{kan} there exists a net  $(m_{\alpha})$ in $ I$ such that $am_{\alpha}-\psi_{\phi}|_{I}(a)m_{\alpha}\rightarrow 0$ and$\psi_{\phi}|_{I}(m_{\alpha})=1$, where $a\in I$.
	Let  $x_{\alpha}\in X$ and $b_{\alpha}\in B$ be
	such that $m_{\alpha}=\left(\begin{array}{cc} 0&x_{\alpha}\\
	0&b_{\alpha}\\
	\end{array}
	\right)$. Then we have $\psi_{\phi}\left(\begin{array}{cc}0&x_{\alpha}\\
	0&b_{\alpha}\\
	\end{array}
	\right)=\phi(b_{\alpha})=1$ and
	\begin{equation}\label{e4-2}
	\left(\begin{array}{cc} 0&x_{0}\\
	0&b_{0}\\
	\end{array}
	\right)\left(\begin{array}{cc}0&x_{\alpha}\\
	0&b_{\alpha}\\
	\end{array}
	\right)-\psi_{\phi}\left(\begin{array}{cc} 0&x_{0}\\
	0&b_{0}\\
	\end{array}
	\right)\left(\begin{array}{cc} 0&x_{\alpha}\\
	0&b_{\alpha}\\\end{array}
	\right)
	\rightarrow 0
	\end{equation} for each $x_{0}\in X$ and $b_{0}\in B.$ Using
	(\ref{e4-2}) we obtain
	$b_{\alpha}b_{0}-\phi(b_{0})b_{\alpha}\rightarrow 0$ and since  $\phi(b_{\alpha})=1$, we see that   $B$
	is left $\phi$-amenable, which contradicts (i).
	
	Now suppose that the statement (ii) holds. Then from (\ref{e4-2}) we have
	$x_{0}b_{\alpha}-\phi(b)x_{\alpha}\rightarrow 0$. By hypothesis from (ii)  there exists a  right
	$\phi$-character $\eta\in X^{*} $. Applying $\eta $ on $x_{0}b_{\alpha}-\phi(b)x_{\alpha}\rightarrow 0$, we have   $\eta(x_{0}b_{\alpha})-\phi(b)\eta(x_{\alpha})\rightarrow 0$ for every $b\in B$ and $x\in X$, which is
	impossible (take $b\in \ker\phi$, implies that $\eta$ is zero), that is, (ii) does not hold.
\end{proof}
It is well-known that,
if $X$ is a compact metric space, then $Lip_{\alpha}(X)$ is unital and the character space $Lip_{\alpha}(X)$ is non-empty, so we have the following corollary.
\begin{cor}
	Suppose that $X$ is a compact metric space. 
		Then $$T=Tri(Lip_{\alpha}(X),Lip_{\alpha}(X),Lip_{\alpha}(X))$$  is not approximately biflat.
\end{cor}
\begin{Theorem}
	Let $T=Tri(A,B,X)$ be a triangular Banach algebra with  $\Delta(B)\neq \emptyset$.
	If one of the
	followings  hold
	\begin{enumerate}
		\item [(i)] B is not left $\phi$-amenable.
		\item [(ii)] $X$ has a right $\phi$-character.
	\end{enumerate}
	Then $T$ is not Johnson pseudo-contractible.
\end{Theorem}
\begin{proof}
	Let   $T$ be Johnson pseudo-contractible. Using a similar argument as in the proof of Lemma \ref{app left}, we can see that
	$T$ is left $\psi_{\phi}$-amenable. Following the proof of Theorem \ref{main} finishes the proof.
\end{proof}
\begin{cor}
Suppose that    $S$ is  the  left zero semigroup. 
	Then $$T=Tri(\ell^{1}(S),\ell^{1}(S),\ell^{1}(S))$$  is not Johnson pseudo-contractible.
\end{cor}
\begin{proof}
	It is known that every semigroup algebra $\ell^{1}(S)$ has a character (for instance, the augmentation character $\phi_{S}$). So  the Banach $(\ell^{1}(S),\ell^{1}(S))$-module $\ell^{1}(S)$ (with natural action)  has a right $\phi_{S}$-character. Thus by previous theorem $T$ is not Johnson psudo-contractible.
\end{proof}
%------------------------------------------------------------------------------------------------------------------------------------------
\begin{small}

\end{small}
\end{document}